\documentclass[12pt]{article}
\usepackage{amssymb,amsmath}
\usepackage{cases}
\usepackage{amsthm}
\usepackage{amsfonts}
\usepackage{graphicx}
\usepackage{cite,color,xcolor}
\usepackage[margin=1 in]{geometry}
\usepackage[colorlinks,citecolor=blue,urlcolor=blue]{hyperref}
\usepackage[utf8]{inputenc}

\newtheorem{theorem}{Theorem}[section]

\newtheorem{lemma}{Lemma}[section]
\newtheorem{proposition}{Proposition}[section]

\newtheorem{remark}{Remark}[section]

\usepackage{txfonts}

\newcommand{\bal}{\begin{align}}
\newcommand{\bbal}{\begin{align*}}
\newcommand{\beq}{\begin{equation}}
\newcommand{\eeq}{\end{equation}}
\newcommand{\bca}{\begin{cases}}
\newcommand{\eca}{\end{cases}}

\newcommand{\pa}{\partial}
\newcommand{\fr}{\frac}

\newcommand{\dd}{\mathrm{d}}

\newcommand{\R}{\mathbb{R}}

\newcommand{\les}{\lesssim}

\newcommand{\f}{\left}
\newcommand{\g}{\right}
\begin{document}

\title{Non-uniform dependence on initial data for the generalized Camassa-Holm  equation in $C^1$}

\author{Yanghai Yu\footnote{E-mail: yuyanghai214@sina.com; lf191110@163.com}\; and Fang Liu\\
\small School of Mathematics and Statistics, Anhui Normal University, Wuhu 241002, China}

\date{\today}

\maketitle\noindent{\hrulefill}

{\bf Abstract:} It is shown in \cite[Adv. Differ. Equ(2017)]{HT} that the Cauchy problem for the generalized Camassa-Holm equation is well-posed in $C^1$ and the data-to-solution map is H\"{o}lder continuous from $C^\alpha$ to $\mathcal{C}([0,T];C^\alpha)$ with $\alpha\in[0,1)$. In this paper, we further show that the data-to-solution map of the generalized Camassa-Holm equation is not uniformly continuous on the initial data in $C^1$. In particular, our result also can be a complement of previous work on the classical Camassa-Holm equation in \cite[Geom. Funct. Anal(2002)]{G02}.


\vskip0mm\noindent{\hrulefill}

\section{Introduction}

In this paper, we consider the Cauchy problem for the generalized Camassa-Holm (g-kbCH) equation as follows
\begin{eqnarray}\label{kb}
        \left\{\begin{array}{ll}
  m_t+u^km_x+bu^{k-1}u_xm=0,\\
  m=u-u_{xx},\\
  u(0, x)=u_0(x), \end{array}\right.
        \end{eqnarray}
where $(x,t)\in \R\times\R^+$ and $(k,b)\in \mathbb{Z}^+\times \R$.
The g-kbCH equation \eqref{kb} can be transformed equivalently into the nonlinear transport type equation
\begin{eqnarray}\label{gch0}
        \left\{\begin{array}{ll}
         u_t+u^ku_x=\mathbf{P}(u),\\
         u(0, x)=u_0(x), \end{array}\right.
        \end{eqnarray}
where
\begin{align}\label{gch1}
\mathbf{P}(u)&=-\pa_x(1-\pa^2_x)^{-1}\left(\frac{b}{k+1} u^{k+1}+\frac{3 k-b}{2} u^{k-1} u_{x}^{2}\right)\nonumber\\
&\quad-(1-\pa^2_x)^{-1}\left(\frac{(k-1)(b-k)}{2}u^{k-2}u_x^3\right).
\end{align}

When $k=1$ and $b=2$, the g-kbCH equation  \eqref{kb} is the famous Camassa-Holm (CH) equation
\begin{equation}\label{CH}
\begin{cases}
\partial_tu+u\pa_xu=\mathbf{P}(u)=-\pa_x(1-\pa^2_x)^{-1}\left(u^2+\fr12(\pa_xu)^2\right), \\
u(x,t=0)=u_0(x),
\end{cases}
\end{equation}
which was originally derived as a bi-Hamiltonian system by Fokas and Fuchssteiner \cite{Fokas1981} in the context of the KdV model and gained prominence after Camassa-Holm \cite{Camassa1993} independently re-derived
it from the Euler equations of hydrodynamics using asymptotic expansions. The CH equation is completely integrable \cite{Camassa1993} with a bi-Hamiltonian structure \cite{Constantin-E,Fokas1981} and infinitely many conservation laws \cite{Camassa1993,Fokas1981}. Also, it admits exact peaked soliton solutions (peakons) of the form $ce^{-|x-ct|}$ with $c>0$, which are orbitally stable \cite{Constantin.Strauss} and models wave breaking (i.e., the solution remains bounded, while its slope becomes unbounded in finite time \cite{Constantin,Escher2,Escher3}).

When $k=1$ and $b=3$, the g-kbCH equation \eqref{kb} becomes the classical Degasperis-Procesi (DP) equation.
The DP equation with a bi-Hamiltonian structure is integrable \cite{DP} and has traveling wave solutions \cite{Lenells,Vakhnenko}. Although the DP equation is similar to the CH equation in several aspects, these two equations are truly
different. One of the novel features of the DP different from the CH equation is that, it has not only peakon solutions \cite{DP} and periodic peakon solutions \cite{YinJFA}, but also shock peakons \cite{Lundmark2007} and the periodic shock waves \cite{Escher}.

When $k=2$ and $b=3$, the g-kbCH equation  \eqref{kb} reduces to the well known Novikov equation \cite{H-H,Ni2011,Novikov2009}
\begin{equation}\label{N}
\begin{cases}
u_t+u^2\pa_xu=-\pa_x(1-\pa^2_x)^{-1}\left(\frac32u(\pa_xu)^2+u^3\right)-\frac12(1-\pa^2_x)^{-1}(\pa_xu)^3,\\
u(x,t=0)=u_0(x).
\end{cases}
\end{equation}
Home-Wang \cite{Home2008} proved that the Novikov equation  with cubic nonlinearity shares similar properties with the CH equation, such as a Lax pair in matrix form, a bi-Hamiltonian structure, infinitely many conserved quantities and peakon solutions given by $u(x, t)=\sqrt{c}e^{-|x-ct|}$.

We should emphasize that, when $b=k+1$, the g-kbCH equation  \eqref{kb} reduces to the following generalized Camassa-Holm-Novikov (gCHN) equation which was proposed by Anco, Silva and Freire \cite{Anco2015}
\begin{eqnarray}\label{eq1}
        \left\{\begin{array}{ll}
         m_t+u^km_x+(k+1)u^{k-1}u_xm=0,\\
         u(0, x)=u_0(x). \end{array}\right.
        \end{eqnarray}
The gCHN equation \eqref{eq1} is an evolution equation with (k+1)-order nonlinearities which can be regarded as a subclass of the g-kbCH equation and, as shown in \cite{Anco2015, Grayshan2013, Himonas2014}, it admits a local conservation law and possesses single peakons of the form $u(x, t)=c^{1/k}e^{-|x-ct|}$ as well as multi-peakon solutions and exhibits wave breaking phenomena (see \cite{Yan2019} and the references therein). Since the $H^1$-norm of solution to the g-kbCH is conserved if and only if $b=k+1$, this means that \eqref{eq1} excludes the DP equation for the case $k=1$ and $b=3$.

In recent years, the question of well-posedness in different spaces for the g-kbCH equation \eqref{gch0}-\eqref{gch1} has become of great interest due to its abundant physical and mathematical properties and a series of achievements have been made in the study of the g-kbCH equation. Using a Galerkin-type approximation scheme, Himonas-Holliman \cite{Himonas2014} established the local well-posedness of the g-kbCH equation in Sobolev space $H^s$ on both the line and  the circle. Zhao et al. \cite{Zhao2014} extended the above well-posedness result to the Besov space $B_{p, r}^s(\mathbb{R})$ with  $s>\max\{1+{1}/{p}, {3}/{2}\}$ and $1\leq p, r\leq \infty$. Subsequently, Chen et al. \cite{Chen2015} solved the critical case for $(s, p, r)=(\frac{3}{2}, 2, 1)$. Motivated by Misio{\l}ek's idea in \cite{G02},
Holmes-Thompson \cite{HT} showed that the Cauchy problem for the g-kbCH equation  \eqref{gch0}-\eqref{gch1} is well-posed in
the space of bounded and continuously differentiable functions on the real line, denoted $C^1(\R)$, and equipped with the norm
$$\|f\|_{C^1}=\|f\|_{L^\infty}+\f\|\frac{\dd f}{\dd x}\g\|_{L^\infty}\quad\text{with}\quad\|f\|_{L^\infty}=\sup_{x\in\R}|f(x)|.$$
More precisely, for initial data $u_0\in C^1$, they obtained a corresponding uniquely
constructed solution $u(x, t) \in \mathcal{C}([-T, T]; \mathcal{C}^1(\R))$, and the solution depends continuously on the initial data. Furthermore, they founded a lifespan estimate, which depends on the size of the initial data. During this lifespan, the solution satisfies the size
estimate
\begin{align}
\|u(t)\|_{C^1}\leq 2\|u_0\|_{C^1}
\end{align}
for $0\leq t\leq T_1$ with $T_1$ depending on $k$ and $\|u_0\|_{C^1}$. Also, they proved the data-to-solution map is H\"{o}lder continuous from $C^\alpha$ to $\mathcal{C}([0,T];C^\alpha)$ with $\alpha\in[0,1)$.

From the PDE's point of view, it is crucial to know if an equation which models a physical phenomenon is well-posed in the
Hadamard's sense: existence, uniqueness, and continuous dependence of the solutions with respect to the initial data. In particular, continuity properties of the solution map is an important part of the well-posedness theory since the lack of continuous dependence would cause incorrect solutions or non meaningful solutions. Furthermore, the non-uniform continuity of data-to-solution map suggests that the local well-posedness cannot be established by the contraction mappings principle since this would imply Lipschitz continuity for the solution map. The method based on the  traveling wave solutions is used to show that non-uniform continuity of data-to-solution map of the periodic CH solutions on initial data in Sobolev
spaces $H^s$ with $s\geq 2$ \cite{Himonas2005} and with $s=1$ on both  the circle and the line \cite{Himonas2007} by  Himonas et al. Using precisely constructed peakon traveling wave solutions, Grayshan and Himonas \cite{Grayshan2013} showed that the solution map of g-kbCH equation is not uniformly continuous in Sobolev spaces $H^s$ with $s<\frac{3}{2}$. We would like to mention that, the peakon solutions belong to any Sobolev space $H^s$ with $s<\frac{3}{2}$, however they barely miss belonging to the space of continuously differentiable functions. Another method is based on the approximate solutions, which differs from the former since it does not require the availability of two actual solutions sequences. This method was used earlier by Koch and Tzvetkov \cite{Koch2005} for the Benjamin-Ono equation, and was further developed by Himonas et al. in the Sobolev spaces $H^s$ with $s>\frac{3}{2}$ for the CH equation  on the line \cite{Himonas2007} and on the circle \cite{Himonas2010}.
By developed a new approximation technique, Li et al. \cite{20jmfm, 20jde,21jmfm,24jde} showed that the solution maps of both the CH and Novikov equations are not uniformly continuous in Besov spaces $B_{p, r}^s$.
For general $k$, Wu-Yu-Xiao \cite{wu2021} proved that the data-to-solution map of the Cauchy problem \eqref{eq1} is not uniformly continuous on the initial data in Besov spaces $B_{p, r}^s(\mathbb{R})$ with $s>\max\{1+{1}/{p},{3}/{2}\}$, $(p, r)\in [1, \infty]\times [1, \infty)$ as well as in critical space $B_{2, 1}^{3/2}(\mathbb{R}).$

To the best of our knowledge, little has been done to show that the above shallow water wave equations exhibit the
non-uniform continuity property of yielding classical solution. We would like to ask that whether or not  H\"{o}lder continuous of the above classical solution obtained by \cite{HT}  can be improved to be Lipschitz continuous. In this paper, we shall prove that the data-to-solution map of the generalized Camassa-Holm equation \eqref{gch0}-\eqref{gch1} as function of the initial data is not uniformly continuous on the initial data in $C^1$. More precisely, we prove the following result.

\begin{theorem}\label{th1} Let $k\in \mathbb{Z}^+$.
Denote $U_R\equiv\f\{u_0\in C^1: \|u_0\|_{C^1}\leq R\g\}$ for any $R>0$.
Then the data-to-solution map of the Cauchy problem \eqref{gch0}-\eqref{gch1}
\begin{equation*}
\mathbf{S}_t:\begin{cases}
U_R \rightarrow \mathcal{C}([0, T] ; C^1) \cap \mathcal{C}^1([0, T] ; C),\\
u_0\mapsto \mathbf{S}_t(u_0),
\end{cases}
\end{equation*}
is not uniformly continuous from any bounded subset $U_R$ in $C^1$ into $\mathcal{C}([0,T];C^1)$. More precisely, there exists two sequences of solutions $\mathbf{S}^u_t(f_n+g_n)$ and $\mathbf{S}_t(f_n)$ such that
\bbal
&\|f_n\|_{C^1}\lesssim 1 \quad\text{and}\quad \lim_{n\rightarrow \infty}\|g_n\|_{C^1}= 0
\end{align*}
but
\bbal
\liminf_{n\rightarrow \infty}\|\mathbf{S}_t(f_n+g_n)-\mathbf{S}_t(f_n)\|_{C^1}\gtrsim t,  \quad \forall \;t\in(0,T_0],
\end{align*}
with small time $T_0$.
\end{theorem}
\begin{remark}
Theorem \ref{th1} demonstrates that dependence of the generalized Camassa-Holm solution on initial data in
$C^1$ can not be better than continuous, which can be a complement of previous result in \cite{HT,G02}.
\end{remark}
\begin{remark}
We should mention that, our Theorem \ref{th1} holds for the Torus case.
\end{remark}
\begin{remark}
Let us make some comments on the idea.

{\bf Case $k\geq3$.}\;
We introduce the following approximate system
\begin{equation}\label{B}
\begin{cases}
\partial_tv+v^k\pa_xv=0,\\
v(0,x)=u_0(x).
\end{cases}
\end{equation}
From now on, we denote the approximate solution of \eqref{B} by $\mathbf{S}^v_{t}(u_0)$. Fortunately, we find a lifespan estimate, which depends on the size of initial data. During this lifespan, the $L^\infty$-norm of approximate solution $\mathbf{S}^v_{t}(u_0)$ remains bounded by the size of the initial data in a purely $L^\infty$ framework.

{\bf Step 1.} We decompose the approximate solution maps as
\bbal
&\mathbf{S}^v_{t}(\underbrace{f_n+g_n}_{=~u^n_0})-\mathbf{S}^v_{t}(f_n)=\underbrace{\mathbf{S}^v_{t}(u^n_0)-u^n_0+t(u^n_{0})^k\pa_xu^n_{0}}_{=~\mathbf{I}_1(u^n_0)}
+\underbrace{f_n-\mathbf{S}^v_{t}(f_n)}_{=~\mathbf{I}_2(f_n)}+g_n-t\underbrace{(f_n+g_n)^k\pa_xu^n_{0}}_{\text{\bf Leading term}}.
\end{align*}
We expect that the {\bf Leading term} plays an essential role since it has a positive lower bound when $n$ is large enough. To deal with the errors estimates $\mathbf{I}_1(u^n_{0})$ and $\mathbf{I}_2(f_n)$  in $C^{1}$, the key point is that, we need to establish the $L^\infty$-estimation of the solution map and finetuned $C^{0,1}$-estimation of $\mathbf{S}^v_{t}(u_0)-u_0$.

{\bf Step 2.} We introduce the errors
\bbal
&w_1:=\mathbf{S}^u_{t}(f_n+g_n)-\mathbf{S}^v_{t}(f_n+g_n)\quad\text{and}\quad w_2:=\mathbf{S}^u_{t}(f_n)-\mathbf{S}^v_{t}(f_n),
\end{align*}
and compare the difference between actual and approximate solutions with same initial data
\bal\label{key}
&\mathbf{S}^u_{t}(f_n+g_n)-\mathbf{S}^u_{t}(f_n)=\mathbf{S}^v_{t}(f_n+g_n)-\mathbf{S}^v_{t}(f_n)+w_1-w_2.
\end{align}
Equality \eqref{key} is crucial since it reduces finding a lower positive
bound for the difference of the actual solution sequences to finding a
lower positive bound for the difference of the approximate solution
sequences. It remains to estimate the $C^1$-norm of Error terms $w_1$ and $w_2$. More precisely, based on the suitable choice of $f_n$ and $g_n$, we shall prove that for a short time $t\in(0,1]$
\bbal
\|w_1\|_{C^{1}}+\|w_2\|_{C^{1}}\les t^2+\varepsilon_n\quad\text{with}\quad \varepsilon_n\rightarrow0 \;\text{as}\; n\rightarrow\infty.
\end{align*}
Combining {\bf Step 1} and {\bf Step 2} enable us that the data-to-solution map of \eqref{gch0}-\eqref{gch1} is not uniformly continuous in
$C^1(\R)$.

{\bf Case $k=1,2$.}\;
We should emphasize that, for the low regularity space $C^1$, it is not possible to proceed as in the previous method due to the lack of the key estimation $\|u\|_{L^\infty}\les \|u_0\|_{L^\infty}$. In fact, due to the appearance of $(\pa_xu)^2$ or $(\pa_xu)^3$, we just obtain the rough estimation $\|u\|_{C^1}\les \|u_0\|_{C^1}$. To bypass the difficulty, we make a new observation: $\|u\|_{L^\infty}\les \|u_0\|_{L^\infty}+\|u_0\|^2_{H^1}$. Next, a key point is to construct initial data $u_0$ satisfying that $\|u_0\|_{L^\infty}$ and $\|u_0\|^2_{H^1}$ possess the same level of size. Finally, we can adopt the standard procedure to obtain the non-uniform continuous of the data-to-solution map for the Camassa-Holm and Novikov equations in $C^1(\R)$.
\end{remark}
{\bf Notations}\;
Given a Banach space $S$, we denote its norm by $\|\cdot\|_{S}$. We also use the simplified notation $\|f_1,\cdots,f_n\|_{S}:=\|f_1\|_{S}+\cdots+\|f_n\|_{S}$. For $I\subset\R$, we denote by $\mathcal{C}(I;S)$ the set of continuous functions on $I$ with values in $S$. For any two positive quantities $X$ and $Y$, $X\lesssim(\gtrsim) Y$ means that there is a uniform positive constant $C$ independent of $X$ and $Y$ such that $X\leq(\geq) CY$. $X\approx Y$ means that $X\lesssim Y$ and $X\gtrsim Y$. Also, we denote $X\ll Y$ if $X\leq \varepsilon Y$ for some sufficiently small constant $\varepsilon>0$. We let $C^1(\R)$ be the Banach space of bounded and continuously differentiable functions which are equipped with the norm $\|f\|_{C^1}=\|f\|_{L^\infty}+\|\pa_xf\|_{L^\infty}$ with $\|f\|_{L^\infty}=\sup_{x\in\R}|f(x)|$. We use $\mathcal{S}(\R)$ and $\mathcal{S}'(\R)$ to denote Schwartz functions and the tempered distributions spaces on $\R$, respectively.
Let us recall that for all $u\in \mathcal{S}'$, the Fourier transform of $u$ is defined by
$
\widehat{u}(\xi)=\int_{\R}e^{-\mathrm{i}x\xi}u(x)\dd x$ for any $\xi\in\R.
$
The inverse Fourier transform of any $g$ is given by
$
(\mathcal{F}^{-1} g)(x)=\frac{1}{2 \pi} \int_{\R} g(\xi) e^{\mathrm{i}x \cdot \xi} \dd \xi.
$

{\bf Organization of our paper.} In Section \ref{sec2}, we establish some technical Lemmas which will be used in the sequel. In Section \ref{sec3} and Section \ref{sec4}, based on a new approximation technique and different construction of initial sequences, we prove Theorem \ref{th1} holds for the real-line case when $k\geq3$ and for the real-line case when $k=1,2$, respectively. In Appendix, based on the traveling wave solutions, we prove Theorem \ref{th1} holds for the Torus case when $k=1$.
\section{Preliminary Lemmas}\label{sec2}
In this section, we establish some useful lemmas.
\subsection{Construction of initial sequences}\label{subs1}
We need to introduce a smooth, radial cut-off function to localize the frequency region. Precisely,
let $\widehat{\phi}\in \mathcal{C}^\infty_0(\mathbb{R})$ be an even, real-valued and non-negative function on $\R$ and satisfy
\begin{numcases}{\widehat{\phi}(\xi)=}
1,&if $|\xi|\leq \frac{1}{4}$,\nonumber\\
0,&if $|\xi|\geq \frac{1}{2}$.\nonumber
\end{numcases}
\begin{remark}\label{re5}
Based on the choice of $\widehat{\phi}$ and the  Fourier-Plancherel formula, we have $\phi(x)=\mathcal{F}^{-1}(\widehat{\phi}(\xi))$.
Thus we find that
$$\|\phi\|_{L^\infty}=\phi(0)=\frac{1}{2\pi}\int_{\R}\widehat{\phi}(\xi)\dd \xi>0.$$
\end{remark}
We establish the following crucial lemma which will be used later on.
\begin{lemma}\label{lm1}
Let $n\gg1$. Define the high-low frequency functions $f_n$ and $g_n$ by
\bbal
&f_n=2^{-n}\phi(x)\cos \left(2^nx\right)\\
& g_n=2^{-\frac{n}{k}}\phi(x).
\end{align*}
Then there exists a positive constant $C=C(\phi)$ which does not depend on $n$ such that
\bal
&\|g_n\|_{L^\infty}+\|\pa_xg_n\|_{L^\infty}+\|\pa_{x}^2g_n\|_{L^\infty}\leq C2^{-\frac{n}{k}},\label{m1}\\
&\|f_n\|_{L^\infty}\leq C2^{-n},\quad\|\pa_xf_n\|_{L^\infty}\leq C,\quad \|\pa_{x}^2f_n\|_{L^\infty}\leq C2^{n},\label{m2}\\
&\left\|(f_n+g_n)^k\pa_{x}^2f_n\right\|_{L^\infty}\geq c_0=:\frac{\phi^{k+1}(0)}{2}. \label{m3}
\end{align}
\end{lemma}
\begin{proof}
\eqref{m1} is obvious and \eqref{m2} holds since the following
\bbal
&\pa_xf_n=-\phi(x)\sin\left(2^nx\right)+2^{-n}\phi'(x)\cos \left(2^nx\right),\\
&\pa_{x}^2f_n=-2^n\f(\phi(x)\cos \left(2^nx\right)+2^{1-n}\phi'(x)\sin \left(2^nx\right)-2^{-2n}\phi''(x)\cos \left(2^nx\right)\g).
\end{align*}
For $n\gg1$, one has
\bbal
\f|(f_n+g_n)^k\pa_{x}^2f_n\g|(x=0)&=2^n(2^{-n}+2^{-\frac{n}{k}})^k\phi^k(0)\f|\phi(0)-2^{-2n}\phi''(0)\g|
\geq \frac{\phi^{k+1}(0)}{2},
\end{align*}
which implies the desired \eqref{m3}.  Thus we finish the proof of Lemma \ref{lm1}.
\end{proof}
\subsection{Estimation of approximate solution}\label{subs2}
\begin{lemma}\label{lm0}
Assume that $u_0\in \mathcal{S}$ with $\|u_0\|_{L^\infty}\ll 1\les\|\pa_xu_0\|_{L^\infty}\ll\|\pa^2_xu_0\|_{L^\infty}$. The data-to-solution map $u_0\mapsto \mathbf{S}^v_t(u_0)$ of the Cauchy problem \eqref{B} satisfies that
\begin{align*}
&\|\mathbf{S}^v_t(u_0)\|_{L^\infty(\R)}=\|u_0\|_{L^\infty(\R)},\\
&\|\pa_x\mathbf{S}^v_t(u_0)\|_{L^{\infty}(\R)}\leq 2\|\pa_xu_0\|_{L^{\infty}(\R)},\\
&\|\pa^2_x\mathbf{S}^v_t(u_0)\|_{L^{\infty}(\R)}\les\|\pa^2_xu_0\|_{L^{\infty}(\R)}+1
\end{align*}
for all $t\in(0,T_2]$ with $T_2=\min\{1,1/(2k\|u_0\|^{k-1}_{L^\infty}\|\partial_xu_0\|_{L^{\infty}})\}$.
\end{lemma}

\begin{proof}
Given a $C^1$-solution $\mathbf{S}^v_t(u_0)$ of Eq.\eqref{B}, we may solve the following ODE to find the flow induced by $[\mathbf{S}^v_t(u_0)]^k$:
\begin{align}\label{ode}
\quad\begin{cases}
\frac{\dd}{\dd t}\phi(t,x)=[\mathbf{S}^v_t(u_0)]^k(t,\phi(t,x)),\\
\phi(0,x)=x.
\end{cases}
\end{align}
Here we set $v(t,x)=\mathbf{S}^v_t(u_0)$ for simplicity. From $\eqref{B}_1$, we get that
\bbal
\frac{\dd}{\dd t}v(t,\phi(t,x))&=v_{t}(t,\phi(t,x))+v_{x}(t,\phi(t,x))\frac{\dd}{\dd t}\phi(t,x)
\\&=\left(v_t+v^k v_x\right)(t, \phi(t, x))=0.
\end{align*}
Integrating the above with respect to time variable yields that
\bbal
v(t,\phi(t,x))=u_0(x).
\end{align*}
Using the fact that the $L^{\infty}$-norm of any function is preserved under the flow $\phi$, then we have
\bal\label{vv}
\|v(t, x)\|_{L^{\infty}(\R)}=\|v(t, \phi(t, x))\|_{L^{\infty}(\R)}=\|u_0(x)\|_{L^{\infty}(\R)}.
\end{align}
Applying $\pa_x$ to $\eqref{B}_1$ yields
\bal\label{u1}
v_{tx}+v^k\pa_{x}^2v=-kv^{k-1}(\pa_xv)^2.
\end{align}
Combining \eqref{ode} and \eqref{u1}, we obtain
\bbal
\frac{\dd}{\dd t}v_x(t,\phi(t,x))&=v_{tx}(t,\phi(t,x))+v_{xx}(t,\phi(t,x))\frac{\dd}{\dd t}\phi(t,x)\nonumber\\
&=\left(v_{tx}+v^kv_{xx}\right)(t,\phi(t,x))=-kv^{k-1}(\pa_xv)^2(t,\phi(t,x)),
\end{align*}
which means that
\bbal
v_x(t,\phi(t,x))=\pa_xu_0(x)-k\int^t_0v^{k-1}(\pa_xv)^2(\tau,\phi(\tau,x))\dd \tau.
\end{align*}
Notice that the $L^{\infty}$-norm of any function is preserved under the flow $\phi$ again, from \eqref{vv}, we get
\begin{align}\label{ml}
\left\|v_x(t,x)\right\|_{L^{\infty}} &\leq\left\|\partial_x u_0\right\|_{L^{\infty}}+k\int_0^t\left(\left\|v(\tau,x)\right\|^{k-1}_{L^{\infty}}\|v_x(\tau,x)\|^2_{L^{\infty}}\right) \dd \tau\nonumber\\
&\leq\left\|\partial_x u_0\right\|_{L^{\infty}}+k\|u_0\|^{k-1}_{L^{\infty}}\int_0^t\|v_x(\tau,x)\|^2_{L^{\infty}} \dd \tau.
\end{align}
Setting
\bbal
\lambda(t):= \left\|\partial_x u_0\right\|_{L^{\infty}}+k\|u_0\|^{k-1}_{L^{\infty}}\int_0^t\|v_x(\tau,x)\|^2_{L^{\infty}} \dd \tau \quad\text{with}\quad \lambda(0):= \left\|\partial_x u_0\right\|_{L^{\infty}},
\end{align*}
then from \eqref{ml}, one has
\bbal
\frac{\dd}{\dd t}\lambda(t)\leq k\left\|u_0\right\|^{k-1}_{L^{\infty}}\lambda^2(t)\quad\Leftrightarrow\quad -\frac{\dd}{\dd t}\left(\frac{1}{\lambda(t)}\right)\leq k\left\|u_0\right\|^{k-1}_{L^{\infty}}.
\end{align*}
Solving the above directly yields for $t\in(0,1/(2k\left\|u_0\right\|^{k-1}_{L^{\infty}}\|\pa_xu_0\|_{L^{\infty}})]$
\bbal
\|\pa_xv(t)\|_{L^{\infty}}\leq \frac{\|\partial_xu_0\|_{L^{\infty}}}{1-tk\left\|u_0\right\|^{k-1}_{L^{\infty}}\|\partial_xu_0\|_{L^{\infty}}}\leq 2\|\partial_xu_0\|_{L^{\infty}}.
\end{align*}
Similarly, we also have
\bbal
\|\pa_x^2v(t)\|_{L^{\infty}}\les \|\partial^2_xu_0\|_{L^{\infty}}+k(k-1)\int^t_0\|v(\tau,x)\|_{L^{\infty}}^{k-2}\|\pa_xv(\tau,x)\|_{L^{\infty}}^{3}\dd \tau.
\end{align*}
This completes the proof of Lemma \ref{lm0}.
\end{proof}

\begin{lemma}\label{lm2} Assume that $u_0\in \mathcal{S}$ with $\|u_0\|_{L^\infty}\ll \|\pa_xu_0\|_{L^\infty}\lesssim 1\ll\|\pa^2_xu_0\|_{L^\infty}$.  The data-to-solution map $u_0\mapsto \mathbf{S}^v_t(u_0)$ of the Cauchy problem \eqref{B} satisfies that for $t\in(0,1]$
\bal
&\f\|\mathbf{S}^v_{t}(u_0)-u_0\g\|_{L^\infty}\leq Ct\|u_0\|^k_{L^{\infty}},\label{y1}\\
&\f\|\pa_x\left(\mathbf{S}^v_{t}(u_0)-u_0\right)\g\|_{L^\infty}\leq Ct\left(\|u_0\|^{k-1}_{L^{\infty}}+\|u_0\|^k_{L^{\infty}}\|\pa_{x}^2u_0\|_{L^{\infty}}\right),\label{y2}\\
&\f\|\pa_{x}^2\left(\mathbf{S}^v_{t}(u_0)-u_0\right)\g\|_{L^\infty}
\leq Ct\left(\|\pa_{x}^2u_0\|_{L^\infty}+\|u_0\|^k_{L^\infty}\|\pa_{x}^3u_0\|_{L^\infty}\right).\label{y3}
\end{align}
\end{lemma}
\begin{proof} By Lemma \ref{lm0}, we know that the solution map $\mathbf{S}^v_{t}(u_0)\in \mathcal{C}([0,T_2];C^1)$ and has common lifespan $T_2$. Moreover, there holds
\bbal
&\|\mathbf{S}_{t}^v(u_0)\|_{L^\infty_{T_2}(L^\infty)}=\|u_0\|_{L^\infty}\quad\text{and}\quad\|\pa_x\mathbf{S}_{t}^v(u_0)\|_{L^\infty_{T_2}(L^\infty)}\les1 .
\end{align*}
By the Newton-Leibniz formula, we have
\bal\label{hy}
\mathbf{S}^v_{t}(u_0)-u_0=-\int^t_0[\mathbf{S}^v_\tau(u_0)]^k\pa_x\mathbf{S}^v_\tau(u_0)\dd\tau,
\end{align}
which gives that
\bal\label{hy1}
\|\mathbf{S}^v_{t}(u_0)-u_0\|_{L^\infty}&\leq \int^t_0\|\mathbf{S}^v_\tau(u_0)\|^k_{L^\infty}\|\pa_x\mathbf{S}^v_\tau(u_0)\|_{L^\infty}\dd \tau
\les t\|u_0\|^k_{L^{\infty}}.
\end{align}
Setting $g=\pa_{x}\left(\mathbf{S}^v_{t}(u_0)-u_0\right)$, then from \eqref{hy}, we deduce
\begin{align*}
\|g(t)\|_{L^\infty}&\leq\int^t_0\|\mathbf{S}^v_\tau(u_0)\|^k_{L^\infty}\|\pa_x^2\mathbf{S}^v_\tau(u_0)\|_{L^\infty}\dd \tau
+k\int^t_0\|\mathbf{S}^v_\tau(u_0)\|^{k-1}_{L^\infty}\|\pa_x\mathbf{S}^v_\tau(u_0)\|^2_{L^\infty}\dd \tau\\&
\les t\f(\|u_0\|^k_{L^{\infty}}\|\pa_x^2u_0\|_{L^{\infty}}+\|u_0\|^{k-1}_{L^{\infty}}\g).
\end{align*}
Setting $h=\pa_{x}^2\left(\mathbf{S}^v_{t}(u_0)-u_0\right)$, then from \eqref{hy}, we have
\begin{align*}
\|h(t)\|_{L^\infty}&\leq \int_0^t\left\|\pa_x^2\f([\mathbf{S}^v_\tau(u_0)]^k\pa_x\mathbf{S}^v_\tau(u_0)\g)\right\|_{L^\infty}\dd \tau\nonumber\\
&\leq Ct\left(1+\|\pa_{x}^2u_0\|_{L^\infty}+\|u_0\|^k_{L^\infty}\|\pa_{x}^3u_0\|_{L^\infty}\right),
\end{align*}
which implies \eqref{y3}. This completes the proof of Lemma \ref{lm2}.
\end{proof}
To obtain the non-uniformly continuous property of data-to-solution map for Eq.\eqref{B}, we need to prove the crucial proposition.
\begin{lemma}\label{lm3}
Assume that $u_0\in \mathcal{S}$ with $\|u_0\|_{L^\infty}\ll \|\pa_xu_0\|_{L^\infty}\lesssim 1\ll\|\pa^2_xu_0\|_{L^\infty}$. Then the data-to-solution map $u_0\mapsto \mathbf{S}^v_t(u_0)$ of the Cauchy problem \eqref{B} satisfies that for $0<t\ll1$
\bbal
\left\|\mathbf{S}^v_{t}(u_0)-u_0+tu_0^k\pa_x u_0\right\|_{C^1}&\les \f(\|u_0\|^k_{L^\infty}\|\pa_{x}^2u_0\|_{L^\infty}+\|u_0\|^{2k}_{L^{\infty}} \|\pa_{x}^3u_0\|_{L^{\infty}}\g)t^2.
\end{align*}
\end{lemma}
\begin{proof} From \eqref{B}, we have
\begin{align}\label{su}
\left\|\mathbf{S}^v_{t}(u_0)-u_0+tu_0^k\pa_x u_0\right\|_{C^1}
\leq&~ \int^t_0\left\|\partial_\tau \mathbf{S}^v_{\tau}(u_0)+u_0^k\pa_x u_0\right\|_{C^1} \dd \tau\nonumber\\
\leq&~ \int^t_0\f\|[\mathbf{S}^v_{\tau}(u_0)]^k\partial_x\mathbf{S}^v_{\tau}(u_0)-u^k_0\partial_xu_0\g\|_{C^1}\dd \tau.
\end{align}
By the binomial theorem, one has
$$[\mathbf{S}^v_{\tau}(u_0)]^{k+1}-u_0^{k+1}=\f(\mathbf{S}^v_{\tau}(u_0)-u_0\g)\cdot\mathbf{F}\f(\mathbf{S}^v_{\tau}(u_0),u_0\g),$$
where
$$\mathbf{F}\f(\mathbf{S}^v_{\tau}(u_0),u_0\g):=[\mathbf{S}^v_{\tau}(u_0)]^{k}+[\mathbf{S}^v_{\tau}(u_0)]^{k-1}\cdot u_0+\cdots+\mathbf{S}^v_{\tau}(u_0)\cdot u_0^{k-1}+u_0^{k}.$$
Thus we have
\begin{align}\label{su1}
&\Big\|[\mathbf{S}^v_{\tau}(u_0)]^k\partial_x\mathbf{S}^v_{\tau}(u_0)-u^k_0\partial_xu_0\Big\|_{C^1}\nonumber\\
=&~\frac{1}{k+1}\f\|\pa_x\f([\mathbf{S}^v_{\tau}(u_0)]^{k+1}-u^{k+1}_0\g)\g\|_{C^1}\nonumber\\
 \les&~\frac{1}{k+1}\left\{\f\|\mathbf{S}^v_{\tau}(u_0)-u_0\g\|_{L^\infty}\f(\f\|\pa_x\mathbf{F}\f(\mathbf{S}^v_{\tau}(u_0),u_0\g)\g\|_{L^\infty}
 +\f\|\pa_x^2\mathbf{F}\f(\mathbf{S}^v_{\tau}(u_0),u_0\g)\g\|_{L^\infty}\g)\right.\nonumber\\
 &\quad+\f\|\pa_x\f(\mathbf{S}^v_{\tau}(u_0)-u_0\g)\g\|_{L^\infty}\f(\f\|\mathbf{F}\f(\mathbf{S}^v_{\tau}(u_0),u_0\g)\g\|_{L^\infty}+
 \f\|\pa_x\mathbf{F}\f(\mathbf{S}^v_{\tau}(u_0),u_0\g)\g\|_{L^\infty}\g)\nonumber\\
 &\quad\left.+\f\|\pa_x^2\f(\mathbf{S}^v_{\tau}(u_0)-u_0\g)\g\|_{L^\infty}\f\|\mathbf{F}\f(\mathbf{S}^v_{\tau}(u_0),u_0\g)\g\|_{L^\infty}\g\}\nonumber\\
 \lesssim&~\f\|\mathbf{S}^v_{\tau}(u_0)-u_0\g\|_{L^\infty}\|\pa_x^2u_0\|_{L^\infty}+\f\|\pa_x\f(\mathbf{S}^v_{\tau}(u_0)-u_0\g)\g\|_{L^\infty}\nonumber\\
 &\quad
+\f\|\pa_x^2\f(\mathbf{S}^v_{\tau}(u_0)-u_0\g)\g\|_{L^\infty}\f\|u_0\g\|^{k}_{L^\infty}.
   \end{align}
Inserting \eqref{su1} into \eqref{su} and using Lemma \ref{lm2}, we complete the proof of Lemma \ref{lm3}.
\end{proof}

\section{Proof of Theorem \ref{th1}: $k\geq3$}\label{sec3}
 In this section, we prove Theorem \ref{th1} holds for the case $k\geq3$.

\subsection{Non-uniform continuous of approximate solution}\label{subs3}
We present the proposition involving the non-uniformly continuous property of data-to-solution map for Eq.\eqref{B}.
\begin{proposition}\label{pro1}
Let $f_n$ and $g_n$ be given in Lemma \ref{lm1}. Then the difference between the data-to-solution maps $f_n+g_n\mapsto \mathbf{S}^v_t(f_n+g_n)$ and  $f_n\mapsto \mathbf{S}^v_t(f_n)$ of the Cauchy problem \eqref{B} satisfies that for $0<t\ll1$, and for some positive constant $c_0$
\bbal
\liminf_{n\rightarrow \infty}\left\|\mathbf{S}^v_{t}(f_n+g_n)-\mathbf{S}^v_{t}(f_n)\right\|_{C^1}\geq  c_0t.
\end{align*}
\end{proposition}
\begin{proof}
Let $u^n_0:=f_n+g_n$. Notice that
\bbal
&\mathbf{S}^v_{t}(u^n_0)=\mathbf{S}^v_{t}(u^n_0)-u^n_0+t(u^n_0)^k\pa_xu^n_0+f_n+g_n-t(f_n+g_n)^k\pa_x(f_n+g_n),
\end{align*}
using the triangle inequality, we deduce that
\bal\label{h0}
\|\mathbf{S}^v_{t}(f_n+g_n)-\mathbf{S}^v_{t}(f_n)\|_{C^1}
&\geq t\f\|(f_n+g_n)^k\pa_x(f_n+g_n)\g\|_{C^1}-\|g_n\|_{C^1}-\f\|\mathbf{S}^v_{t}(f_n)-f_n\g\|_{C^1}\nonumber\\
&\quad-\f\|\mathbf{S}^v_{t}(u^n_0)-u^n_0+t(u^n_0)^k\pa_xu^n_0\g\|_{C^1}.
\end{align}
It is not difficult to find that
\begin{align}\label{yyy1}
\f\|(f_n+g_n)^k\partial_x(f_n+g_n)\g\|_{C^1}&\geq \f\|(f_n+g_n)^k\partial_xf_n\g\|_{C^1}
-\f\|(f_n+g_n)^k\partial_xg_n\g\|_{C^1}\nonumber\\
&\geq c_0-\f\|(f_n+g_n)^k\g\|_{C^1}\f(\f\|\partial_xf_n\g\|_{L^\infty}+\f\|\partial_xg_n\g\|_{C^1}\g)\nonumber\\
&\geq c_0
-C2^{-\frac {(k-1)n}k},
\end{align}
where we have used Lemma \ref{lm1} and
\bbal
& \f\|(f_n+g_n)^k\g\|_{C^1}\les\f\|(f_n+g_n)^k\g\|_{L^\infty}+\f\|\pa_x\f((f_n+g_n)^k\g)\g\|_{L^\infty}\les2^{-\frac {(k-1)n}k}.
\end{align*}
Using Lemma  \ref{lm2} with $u_0=f_n$, and by Lemma \ref{lm1}, we obtain
\bal\label{yy2}
\f\|\mathbf{S}^v_{t}(f_n)-f_n\g\|_{C^1}\les t2^{-(k-1)n}.
\end{align}
Using Lemma  \ref{lm3} with $u_0=f_n+g_n$, and by Lemma \ref{lm1}, we obtain
\bal\label{yy3}
\f\|\mathbf{S}^v_{t}(u^n_0)-u^n_0+t(u^n_0)^k\pa_xu^n_0\g\|_{C^1}\les t^2.
\end{align}
Inserting the above \eqref{yyy1}-\eqref{yy3} into \eqref{h0}, we deduce  that
\bbal
\liminf_{n\rightarrow \infty}\|\mathbf{S}^v_t(f_n+g_n)-\mathbf{S}^v_t(f_n)\|_{C^1}\gtrsim t\quad\text{for} \ t \ \text{small enough}.
\end{align*}
This completes the proof of Proposition \ref{pro1}.
\end{proof}
\subsection{Estimation of the $C^1$-norm of Errors}\label{subs4}
To obtain the non-uniformly continuous property for the actual solution, we need to estimate the $C^1$-norm of the difference between approximate and actual solution with the same initial data.
\begin{proposition}\label{pro2}
Assume that $u_0\in \mathcal{S}$ with $\|u_0\|_{L^\infty}\ll \|\pa_xu_0\|_{L^\infty}\lesssim 1\ll\|\pa^2_xu_0\|_{L^\infty}$. Then the difference between actual $\mathbf{S}^u_t(u_0)$ of the Cauchy problem \eqref{gch0}-\eqref{gch1} and  approximate solution $\mathbf{S}^v_t(u_0)$ of the Cauchy problem \eqref{B}, satisfies that for $0<t\ll1$
\bbal
\left\|\mathbf{S}^u_{t}(u_0)-\mathbf{S}^v_{t}(u_0)\right\|_{C^1}&\les t^2\left\|u_0\right\|_{L^{\infty}}^{k}\left\|\pa_x^2u_0\right\|_{L^{\infty}}+\|u_0\|^{k-2}_{L^{\infty}}.
\end{align*}
\end{proposition}
\begin{proof}
For the sake of simplicity, we set $\mathbf{w}:=\mathbf{S}^u_{t}(u_0)-\mathbf{S}^v_{t}(u_0)$. Then we have
\begin{equation}\label{BC0}
\begin{cases}
\partial_t\mathbf{w}+[\mathbf{S}^u_{t}(u_0)]^k\pa_x\mathbf{w}=-\f([\mathbf{S}^u_{t}(u_0)]^k-[\mathbf{S}^v_{t}(u_0)]^k\g)\pa_x\mathbf{S}^v_{t}(u_0)+\mathbf{P}(\mathbf{S}^u_{t}(u_0)),\\
\mathbf{w}(x,t=0)=0.
\end{cases}
\end{equation}
Given a $C^1$-solution $\mathbf{S}^u_{t}(u_0)$ of the Cauchy problem \eqref{gch0}-\eqref{gch1}, we may solve the following ODE to find the flow induced by $[\mathbf{S}^u_{t}(u_0)]^k$:
\begin{align}\label{ode1}
\quad\begin{cases}
\frac{\dd}{\dd t}\psi(t,x)=[\mathbf{S}^u_{t}(u_0)]^k(t,\psi(t,x)),\\
\psi(0,x)=x.
\end{cases}
\end{align}
Now we consider the equation \eqref{BC0} along the Lagrangian flow-map $\psi(t,x)$ associated to $[\mathbf{S}^u_{t}(u_0)]^k$. From \eqref{BC0} and \eqref{ode1}, it follows that
$$\mathbf{w}(t,\psi(t,x))=
-\int_0^t\f\{\f([\mathbf{S}^u_{\tau}(u_0)]^k-[\mathbf{S}^v_{\tau}(u_0)]^k\g)\pa_x\mathbf{S}^v_{\tau}(u_0)-\mathbf{P}(\mathbf{S}^u_{\tau}(u_0))\g\}(\tau,\psi(\tau,x))\dd \tau.$$
Notice that the $L^{\infty}$-norm of any function is preserved under the flow $\psi$ again, we have
\begin{align}\label{FY}
\left\|\mathbf{w}(t,x)\right\|_{L^{\infty}} &\leq\int_0^t\underbrace{\left\|\f([\mathbf{S}^u_{\tau}(u_0)]^k-[\mathbf{S}^v_{\tau}(u_0)]^k\g)\pa_x\mathbf{S}^v_{\tau}(u_0)\right\|_{L^{\infty}} }_{=:I_1} +\underbrace{\left\|\mathbf{P}(\mathbf{S}^u_{\tau}(u_0))\right\|_{L^{\infty}}}_{=:I_2} \dd \tau.
\end{align}
Note that $\mathbf{S}^u_{\tau}(u_0)=\mathbf{w}+\mathbf{S}^v_{\tau}(u_0)$ and by the binomial theorem, one has
\begin{align}\label{ly}
[\mathbf{S}^u_{\tau}(u_0)]^k-[\mathbf{S}^v_{\tau}(u_0)]^k=\sum_{i=1}^kC_k^i\mathbf{w}^i[\mathbf{S}^v_{\tau}(u_0)]^{k-i}.
\end{align}
The first term $I_1$ can be estimated as follows
\begin{align*}
I_1 &\les\left\|[\mathbf{S}^u_{\tau}(u_0)]^k-[\mathbf{S}^v_{\tau}(u_0)]^k\right\|_{L^{\infty}}\les\left\|\mathbf{w}(\tau,x)\right\|_{L^{\infty}}.
\end{align*}
Using the following estimate:
$$\left\|\partial_x\left(1-\partial^2_{x}\right)^{-1} f\right\|_{L^{\infty}}=\left\|\partial_xG*f\right\|_{L^{\infty}} \leq\|f\|_{L^{\infty}}\quad\text{where}\quad G(x)=\fr12e^{-|x|},
$$
we easily estimate the second term $I_2$ as follows
\begin{align*}
I_2 &\les\left\|[\mathbf{S}^u_{\tau}(u_0)]^{k+1}-[\mathbf{S}^v_{\tau}(u_0)]^{k+1}\right\|_{L^{\infty}}+\left\|\mathbf{S}^v_{\tau}(u_0)\right\|^{k+1}_{L^{\infty}}\\
&\quad+ \left\|[\mathbf{S}^u_{\tau}(u_0)]^{k-1}-[\mathbf{S}^v_{\tau}(u_0)]^{k-1}\right\|_{L^{\infty}}+\left\|\mathbf{S}^v_{\tau}(u_0)\right\|^{k-1}_{L^{\infty}}\\
&\quad+\left\|[\mathbf{S}^u_{\tau}(u_0)]^{k-2}-[\mathbf{S}^v_{\tau}(u_0)]^{k-2}\right\|_{L^{\infty}}+\left\|\mathbf{S}^v_{\tau}(u_0)\right\|^{k-2}_{L^{\infty}}\\
&\les\left\|\mathbf{w}(\tau,x)\right\|_{L^{\infty}}+\|u_0\|^{k-2}_{L^{\infty}}.
\end{align*}
Inserting the above into \eqref{FY} yields that
\begin{align*}
\left\|\mathbf{w}(t,x)\right\|_{L^{\infty}} &\les t\|u_0\|^{k-2}_{L^{\infty}}+\int_0^t\left\|\mathbf{w}(\tau,x)\right\|_{L^{\infty}}\dd \tau,
\end{align*}
and using Gronwall's inequality yields that
\begin{align}\label{w1}
\left\|\mathbf{w}(t,x)\right\|_{L^{\infty}} &\leq  Ct\|u_0\|^{k-2}_{L^{\infty}}.
\end{align}
Applying $\pa_x$ to \eqref{BC0} yields
\begin{equation*}
\begin{cases}
\partial_t\mathbf{w}_x+[\mathbf{S}^u_{t}(u_0)]^k\pa_x\mathbf{w}_x=-\f([\mathbf{S}^u_{t}(u_0)]^k-[\mathbf{S}^v_{t}(u_0)]^k\g)\pa_x^2\mathbf{S}^v_{t}(u_0)+\pa_x\mathbf{P}(\mathbf{S}^u_{t}(u_0))\\
\qquad-k[\mathbf{S}^u_{t}(u_0)]^{k-1}\pa_x\mathbf{S}^u_{t}(u_0)\pa_x\mathbf{w}-\pa_x\f([\mathbf{S}^u_{t}(u_0)]^k-[\mathbf{S}^v_{t}(u_0)]^k\g)\pa_x\mathbf{S}^v_{t}(u_0),\\
\mathbf{w}_x(x,t=0)=0.
\end{cases}
\end{equation*}
Similarly, we have
\begin{align}\label{BC1}
\left\|\mathbf{w}_x(t,x)\right\|_{L^{\infty}} &\leq\int_0^t\underbrace{\left\|\f([\mathbf{S}^u_{t}(u_0)]^k-[\mathbf{S}^v_{t}(u_0)]^k\g)\pa_x^2\mathbf{S}^v_{t}(u_0)\right\|_{L^{\infty}}}_{=:J_1}
+\underbrace{\left\|\pa_x\mathbf{P}(\mathbf{S}^u_{\tau}(u_0))\right\|_{L^{\infty}}}_{=:J_2} \dd \tau\nonumber\\
&\quad+\int_0^t\underbrace{\left\|k[\mathbf{S}^u_{t}(u_0)]^{k-1}\pa_x\mathbf{S}^u_{t}(u_0)\pa_x\mathbf{w}\right\|_{L^{\infty}}}_{=:J_3}\nonumber\\
&\quad+\underbrace{\left\|\pa_x\f([\mathbf{S}^u_{t}(u_0)]^k-[\mathbf{S}^v_{t}(u_0)]^k\g)\pa_x\mathbf{S}^v_{t}(u_0)\right\|_{L^{\infty}}}_{=:J_4} \dd \tau.
\end{align}
We estimate the above four terms as follows
\begin{align*}
J_1&\les\left\|[\mathbf{S}^u_{\tau}(u_0)]^k-[\mathbf{S}^v_{\tau}(u_0)]^k\right\|_{L^{\infty}}\left\|\pa_x^2u_0\right\|_{L^{\infty}}\quad \text{by} \;\eqref{ly}\\
&\les\sum_{i=1}^kC_k^i\left\|\mathbf{w}(\tau,x)\right\|_{L^{\infty}}^i\|u_0\|_{L^{\infty}}^{k-i}\left\|\pa_x^2u_0\right\|_{L^{\infty}}\quad \text{by} \;\eqref{w1}\\
&\les\sum_{i=1}^kC_k^i\tau^i\|u_0\|_{L^{\infty}}^{k}\left\|\pa_x^2u_0\right\|_{L^{\infty}}\les \tau\|u_0\|_{L^{\infty}}^{k}\left\|\pa_x^2u_0\right\|_{L^{\infty}},\\
J_2&\les\left\|\mathbf{w}(\tau,x)\right\|_{L^{\infty}}+\|u_0\|^{k-2}_{L^{\infty}}\les\|u_0\|^{k-2}_{L^{\infty}},\\
J_3&\les\left\|\mathbf{w}_x(\tau,x)\right\|_{L^{\infty}},\\
J_4&\les\left\|\mathbf{w}_x(\tau,x)\right\|_{L^{\infty}}+\left\|\mathbf{w}(\tau,x)\right\|_{L^{\infty}}\les\left\|\mathbf{w}_x(\tau,x)\right\|_{L^{\infty}}+\|u_0\|^{k-2}_{L^{\infty}}.
\end{align*}
Inserting the above into \eqref{BC1} yields that
\begin{align*}
\left\|\mathbf{w}_x(t,x)\right\|_{L^{\infty}} &\les
t^2\left\|u_0\right\|_{L^{\infty}}^{k}\left\|\pa_x^2u_0\right\|_{L^{\infty}}
+t\|u_0\|^{k-2}_{L^{\infty}}+\int_0^t\left\|\mathbf{w}_x(\tau,x)\right\|_{L^{\infty}}\dd \tau,
\end{align*}
and using Gronwall's inequality, yields that
\begin{align}\label{w2}
\left\|\mathbf{w}_x(t,x)\right\|_{L^{\infty}}  &\les t^2\left\|u_0\right\|_{L^{\infty}}^{k}\left\|\pa_x^2u_0\right\|_{L^{\infty}}+\|u_0\|^{k-2}_{L^{\infty}}.
\end{align}
Combining \eqref{w1} and \eqref{w2} gives the desired result of Proposition \ref{pro2}.
\end{proof}
\subsection{Non-uniform continuous of actual solution}\label{subs5}
With Propositions \ref{pro1}-\ref{pro2} in hand, we can prove Theorem \ref{th1}.\\
{\bf Behavior at time $t=0$.}\quad We set $u^n_0=f_n+g_n$. Obviously, we have
\bbal
\|u^n_0-f_n\|_{C^1}=\|g_n\|_{C^1}\leq C2^{-\frac{n}{k}}\quad \Rightarrow\quad
\lim_{n\to\infty}\|u^n_0-f_n\|_{C^1}=0.
\end{align*}
{\bf Behavior at time $t>0$.}\quad  We decompose the solution $\mathbf{S}^u_{t}(u^n_0)$ and $\mathbf{S}^u_{t}(f_n)$ as
\bbal
&\mathbf{S}^u_{t}(f_n+g_n)=\mathbf{S}^v_{t}(f_n+g_n)+w_1\quad\text{and}\nonumber\\
&\mathbf{S}^u_{t}(f_n)=\mathbf{S}^v_{t}(f_n)+w_2.
\end{align*}
Using Proposition \ref{pro2} with $u_0=f_n+g_n$ and $u_0=f_n$, respectively, and by Lemma \ref{lm1}, we obtain
\bbal
\|w_1\|_{C^1}\les t^2+2^{-\frac{(k-2)n}{k}}\quad\text{and}\quad\|w_2\|_{C^1}\les t^2+2^{-kn}.
\end{align*}
Using the triangle inequality and Proposition \ref{pro1}, we deduce that
\bbal
\liminf_{n\rightarrow \infty}\|\mathbf{S}^u_{t}(f_n+g_n)-\mathbf{S}^u_{t}(f_n)\|_{C^1}
=&~\liminf_{n\rightarrow \infty}\|\mathbf{S}^v_{t}(f_n+g_n)-\mathbf{S}^v_{t}(f_n)+w_1-w_2\|_{C^1}\nonumber\\
\geq&~\liminf_{n\rightarrow \infty}\|\mathbf{S}^v_{t}(f_n+g_n)-\mathbf{S}^v_{t}(f_n)\|_{C^1}-Ct^2\nonumber\\
\geq&~c_0 t-Ct^{2}\geq \frac{c_0}{2} t.
\end{align*}
This completes the proof of Theorem \ref{th1}.{\hfill $\square$}

\begin{remark} Following the same argument as above, we can prove Theorem \ref{th1} holds for the equation \eqref{gch0}-\eqref{gch1} with $k\geq3$ on the circle.
\end{remark}
\section{Proof of Theorem \ref{th1}: $k=1,2$}\label{sec4}
 In this section, we just prove Theorem \ref{th1} holds for the case $k=1$ since the case $k=2$ can be done with minor modification.
Firstly, we need to modify the high frequency function and establish the following crucial lemma.
\begin{lemma}\label{lmm1}
Let $n\gg1,\,\lambda\gg1$ and $\phi$ be given in Remark \ref{re5}. Define the high-low frequency functions $f_n$ and $g_n$ by
\bbal
&f_n=2^{-n}\phi(2^nx)\cos \left(\lambda2^nx\right),\\
& g_n=2^{-n}\phi(x).
\end{align*}
Then there exists a positive constant $C=C(\phi)$ which does not depend on $n$ such that
\bal
&\|g_n\|_{L^\infty}+\|\pa_xg_n\|_{L^\infty}+\|\pa_{x}^2g_n\|_{L^\infty}\leq C2^{-n},\label{mf1}\\
&\|f_n\|_{L^\infty}\leq C2^{-n},\quad\|\pa_x^kf_n\|_{L^\infty}\leq C2^{(k-1)n},\quad\text{for}\quad k\in\{1,2,3\},\label{mf2}\\
&\|f_n\|_{H^1}\leq C2^{-\frac{n}{2}},\quad\|g_n\|_{H^1}=2^{-n}\f\|\phi\g\|_{H^1},\label{mf3}\\
&\left\|g_n\pa_{x}^2f_n\right\|_{L^\infty}\geq \frac{\lambda^2\phi^{2}(0)}{2}. \label{mf4}
\end{align}
\end{lemma}
\begin{proof}
\eqref{mf1} is obvious and \eqref{mf2} holds since the following
\bbal
&\pa_xf_n=-\lambda\phi(2^nx)\sin\left(\lambda2^nx\right)+\phi'(2^nx)\cos \left(\lambda2^nx\right),\\
&\pa_{x}^2f_n=-2^n\f(\lambda^2\phi(2^nx)\cos \left(\lambda2^nx\right)+2\lambda\phi'(2^nx)\sin \left(\lambda2^nx\right)-\phi''(2^nx)\cos \left(\lambda2^nx\right)\g).
\end{align*}
By easy computations, we obtain
\bbal
&\f\|f_n\g\|^2_{L^2}\leq2^{-2n}\int_{\R}\f|\phi(2^nx)\g|^2\dd x
\leq 2^{-2n}\f\|\phi\g\|^2_{H^1},\\
&\f\|\pa_xf_n\g\|^2_{L^2}\leq4\lambda^2\int_{\R}\f|\phi(2^nx)\g|^2+\f|\phi'(2^nx)\g|^2\dd x
\leq 4\lambda^22^{-n}\f\|\phi\g\|^2_{H^1},
\end{align*}
which gives \eqref{mf3}.
For $\lambda \gg1$, one has
\bbal
\f|g_n\pa_{x}^2f_n\g|(x=0)&=\phi(0)\f|\lambda^2\phi(0)-\phi''(0)\g|
\geq \frac{\lambda^2\phi^{2}(0)}{2},
\end{align*}
which implies the desired \eqref{mf4}.  Thus we finish the proof of Lemma \ref{lmm1}.
\end{proof}

\begin{lemma}\label{lmm0}
Assume that $u_0\in \mathcal{S}$ with $\|u_0\|_{L^\infty}\ll \|\pa_xu_0\|_{L^\infty}\lesssim 1\ll\|\pa^2_xu_0\|_{L^\infty}$.  The data-to-solution map $u_0\mapsto \mathbf{S}_t(u_0)$ of the Cauchy problem \eqref{CH} satisfies that for $t\in(0,1]$
\begin{align*}
&\|\mathbf{S}_t(u_0)\|_{L^\infty(\R)}\les\|u_0\|_{L^\infty(\R)}+\|u_0\|_{H^1}^2,\\
&\|\pa_x\mathbf{S}_t(u_0)\|_{L^{\infty}(\R)}\les\|\pa_xu_0\|_{L^{\infty}(\R)},\\
&\|\pa^2_x\mathbf{S}_t(u_0)\|_{L^{\infty}(\R)}\les\|\pa^2_xu_0\|_{L^{\infty}(\R)}.
\end{align*}
\end{lemma}

\begin{proof} We set $u(t,x)=\mathbf{S}_t(u_0)$ for simplicity.
Given a $C^1$-solution $u(t,x)$ of Eq.\eqref{CH}, we may solve the following ODE to find the flow induced by $u$:
\begin{align*}
\quad\begin{cases}
\frac{\dd}{\dd t}\phi(t,x)=u(t,\phi(t,x)),\\
\phi(0,x)=x.
\end{cases}
\end{align*}
From $\eqref{CH}_1$, we get that
\bbal
\frac{\dd}{\dd t}u(t,\phi(t,x))&=u_{t}(t,\phi(t,x))+u_{x}(t,\phi(t,x))\frac{\dd}{\dd t}\phi(t,x)
=\mathbf{P}(u)(t,\phi(t,x)).
\end{align*}
Integrating the above with respect to time variable yields that
\bbal
u(t,\phi(t,x))=u_0(x)+\int_0^t\mathbf{P}(u)(\tau,\phi(t,x))\dd \tau.
\end{align*}
Using the fact that the $L^{\infty}$-norm of any function is preserved under the flow $\phi$, then we have
\bal\label{hlf}
\|u(t, x)\|_{L^{\infty}(\R)}=\|u(t, \phi(t, x))\|_{L^{\infty}(\R)}\leq\|u_0(x)\|_{L^{\infty}(\R)}+\int_0^t\|\mathbf{P}(u(\tau, x))\|_{L^{\infty}(\R)}\dd \tau.
\end{align}
Notice that the fact $\partial_x^{-1}f(x)=\int_0^xf(y)\dd y$ and  the following key estimate:
\begin{align*}
\left\|\partial_x\left(1-\partial^2_{x}\right)^{-1} f\right\|_{L^{\infty}}&=\left\|\partial^2_x\left(1-\partial^2_{x}\right)^{-1} \f(\int_{-\infty}^{x}f\dd y\g)\right\|_{L^{\infty}}
\leq2\int_{-\infty}^{+\infty}|f(x)|\dd x,
\end{align*}
we easily estimate the last term as follows
\begin{align}\label{lf11}
\f\|\mathbf{P}\f(u\g)\g\|_{L^{\infty}} &\les\int_{\R}|u|^2+\fr12|\pa_xu|^2\dd x\les \|u_0\|_{H^1}^2,
\end{align}
where we have used the fact that the $H^1$-norm of solution to the CH equation is conserved, i.e., $\|u\|_{H^1}=\|u_0\|_{H^1}$.

Inserting \eqref{lf11} into \eqref{hlf} yields the first inequality.
Following the same procedure with that of Lemma \ref{lm0}, we can obtain the remaining inequalities. We complete the proof of Lemma \ref{lmm0}.
\end{proof}

\begin{lemma}\label{lmm2} Assume that $u_0\in \mathcal{S}$ with $\|u_0\|_{L^\infty}\ll \|\pa_xu_0\|_{L^\infty}\lesssim 1\ll\|\pa^2_xu_0\|_{L^\infty}$.  The data-to-solution map $u_0\mapsto \mathbf{S}_t(u_0)$ of the Cauchy problem \eqref{CH} satisfies that for $t\in(0,1]$
\bal
&\f\|\mathbf{S}_{t}(u_0)-u_0\g\|_{L^\infty}\les t\f(\|u_0\|_{L^{\infty}}+\|u_0\|^2_{H^{1}}\g),\label{y1y}\\
&\f\|\pa_x\left(\mathbf{S}_{t}(u_0)-u_0\right)\g\|_{L^\infty}\leq Ct\left(1+\big(\|u_0\|_{L^{\infty}}+\|u_0\|^2_{H^{1}}\big)\|\pa_{x}^2u_0\|_{L^{\infty}}\right),\label{y2y}\\
&\f\|\pa_{x}^2\left(\mathbf{S}_{t}(u_0)-u_0\right)\g\|_{L^\infty}
\leq Ct\left(\|\pa_{x}^2u_0\|_{L^\infty}+\big(\|u_0\|_{L^{\infty}}+\|u_0\|^2_{H^{1}}\big)\|\pa_{x}^3u_0\|_{L^\infty}\right).\label{y3y}
\end{align}
\end{lemma}
\begin{proof} By Lemma \ref{lm0}, we know that the solution map $\mathbf{S}_{t}(u_0)\in \mathcal{C}([0,T];C^1)$ and has common lifespan $T$. Moreover, there holds
\bbal
&\|\mathbf{S}_{t}(u_0)\|_{L^\infty_{T}(C^1)}\les\|u_0\|_{C^{1}}\les 1.
\end{align*}
By the Newton-Leibniz formula, we have
\bal\label{hyy}
\mathbf{S}_{t}(u_0)-u_0=-\int^t_0\mathbf{S}_\tau(u_0)\pa_x\mathbf{S}_\tau(u_0)\dd\tau+\int^t_0\mathbf{P}\f(\mathbf{S}_\tau(u_0)\g)\dd\tau.
\end{align}
Using \eqref{lf11} and \eqref{hyy} yields
\bbal
\|\mathbf{S}_{t}(u_0)-u_0\|_{L^\infty}&\leq \int^t_0\|\mathbf{S}_\tau(u_0)-u_0\|_{L^\infty}\|\pa_x\mathbf{S}_\tau(u_0)\|_{L^\infty}\dd \tau+t\f(\|u_0\|_{L^{\infty}}+\|u_0\|^2_{H^{1}}\g)\\
&\leq \int^t_0\|\mathbf{S}_\tau(u_0)-u_0\|_{L^\infty}\dd \tau+t\f(\|u_0\|_{L^{\infty}}+\|u_0\|^2_{H^{1}}\g),
\end{align*}
from which and Gronwall's inequality, one has \eqref{y1y}.

Setting $g=\pa_{x}\left(\mathbf{S}_{t}(u_0)-u_0\right)$, then from \eqref{hyy}, we deduce
\begin{align*}
\|g(t)\|_{L^\infty}&\leq\int^t_0\|\mathbf{S}_\tau(u_0)\|_{L^\infty}\|\pa_x^2\mathbf{S}_\tau(u_0)\|_{L^\infty}\dd \tau
+\int^t_0\|\pa_x\mathbf{S}_\tau(u_0)\|^2_{L^\infty}\dd \tau+\int^t_0\|\pa_x\mathbf{P}\f(\mathbf{S}_\tau(u_0)\g)\|_{L^\infty}\dd\tau\\&
\les t\f(\big(\|u_0\|_{L^{\infty}}+\|u_0\|^2_{H^{1}}\big)\|\pa_x^2u_0\|_{L^{\infty}}+1\g).
\end{align*}
Setting $h=\pa_{x}^2\left(\mathbf{S}_{t}(u_0)-u_0\right)$, then from \eqref{hyy}, we have
\begin{align*}
\|h(t)\|_{L^\infty}&\leq \int_0^t\left\|\pa_x^2\f(\mathbf{S}_\tau(u_0)\pa_x\mathbf{S}_\tau(u_0)\g)\right\|_{L^\infty}\dd \tau+\int_0^t\left\|\pa_x^2\mathbf{P}\f(\mathbf{S}_\tau(u_0)\g)\right\|_{L^\infty}\dd \tau\nonumber\\
&\les t\left(1+\|\pa_{x}^2u_0\|_{L^\infty}+\big(\|u_0\|_{L^{\infty}}+\|u_0\|^2_{H^{1}}\big)\|\pa_{x}^3u_0\|_{L^\infty}\right),
\end{align*}
which implies \eqref{y3y}. This completes the proof of Lemma \ref{lmm2}.
\end{proof}
To obtain the non-uniformly continuous property of data-to-solution map for Eq.\eqref{CH}, we need to prove the crucial Lemma.
\begin{lemma}\label{lmm3}
Assume that $u_0\in \mathcal{S}$ with $\|u_0\|_{L^\infty}\ll \|\pa_xu_0\|_{L^\infty}\lesssim 1\ll\|\pa^2_xu_0\|_{L^\infty}$. Then the data-to-solution map $u_0\mapsto \mathbf{S}_t(u_0)$ of the Cauchy problem \eqref{CH} satisfies that for $0<t\ll1$
\bbal
\left\|\mathbf{E}(u_0)\right\|_{C^1}&\les \f(\f\|u_0\g\|_{L^\infty}+\|u_0\|_{H^1}^2\g)\f(\|\pa_{x}^2u_0\|_{L^\infty}+\big(\|u_0\|_{L^{\infty}}+\|u_0\|^2_{H^{1}}\big) \|\pa_{x}^3u_0\|_{L^{\infty}}\g)t^2,
\end{align*}
where we denote
$$\mathbf{E}(u_0,t):=\mathbf{S}_{t}(u_0)-u_0+tu_0\pa_x u_0-\int^t_0\mathbf{P}\f(\mathbf{S}_{\tau}(u_0)\g)\dd \tau.$$
\end{lemma}
\begin{proof} From \eqref{hyy}, we have
\begin{align*}
\mathbf{E}(u_0,t)
=\int^t_0\left(u_0\partial_xu_0-\mathbf{S}_{\tau}(u_0)\partial_x\mathbf{S}_{\tau}(u_0)\right)\dd \tau,
\end{align*}
which gives that
\begin{align*}
\left\|\mathbf{E}(u_0,t)\right\|_{C^1}
&\leq \int^t_0\f\|\mathbf{S}_{\tau}(u_0)\partial_x\mathbf{S}_{\tau}(u_0)-u_0\partial_xu_0\g\|_{C^1}\dd \tau\nonumber\\
&\leq\int^t_0\f\|\pa_x[\f(\mathbf{S}_{\tau}(u_0)-u_0\g)\f(\mathbf{S}_{\tau}(u_0)+u_0\g)],\;\pa^2_x[\f(\mathbf{S}_{\tau}(u_0)-u_0\g)\f(\mathbf{S}_{\tau}(u_0)+u_0\g)]\g\|_{L^\infty}\dd \tau\nonumber\\
&\les\int^t_0\f\|\mathbf{S}_{\tau}(u_0)-u_0\g\|_{L^\infty}\f\|\pa_x\f(\mathbf{S}_{\tau}(u_0)+u_0\g),\;\pa_x^2\f(\mathbf{S}_{\tau}(u_0)+u_0\g)\g\|_{L^\infty}\dd \tau\nonumber\\
&\quad+\int^t_0\f\|\pa_x\f(\mathbf{S}_{\tau}(u_0)-u_0\g)\g\|_{L^\infty}\f\|\mathbf{S}_{\tau}(u_0)+u_0,\;\pa_x\f(\mathbf{S}_{\tau}(u_0)+u_0\g)\g\|_{L^\infty}\dd \tau\nonumber\\
 &\quad+\int^t_0\f\|\pa_x^2\f(\mathbf{S}_{\tau}(u_0)-u_0\g)\g\|_{L^\infty}\f\|\mathbf{S}_{\tau}(u_0)+u_0\g\|_{L^\infty}\dd \tau\nonumber\\
 &\les\int^t_0\f\|\mathbf{S}_{\tau}(u_0)-u_0\g\|_{L^\infty}\|\pa_x^2u_0\|_{L^\infty}\dd \tau+\int^t_0\f\|\pa_x\f(\mathbf{S}_{\tau}(u_0)-u_0\g)\g\|_{L^\infty}\dd \tau\nonumber\\
 &\quad
+\int^t_0\f\|\pa_x^2\f(\mathbf{S}_{\tau}(u_0)-u_0\g)\g\|_{L^\infty}\f(\f\|u_0\g\|_{L^\infty}+\|u_0\|_{H^1}^2\g)\dd \tau.
\end{align*}
Using Lemma \ref{lmm2} to the  above, we complete the proof of Lemma \ref{lmm3}.
\end{proof}

Finally, we present the proposition which leads to the non-uniformly continuous property of data-to-solution map for for the CH equation \eqref{CH}.
\begin{proposition}\label{prop}
Let $f_n$ and $g_n$ be given in Lemma \ref{lmm1}. Then the difference between the data-to-solution maps $f_n+g_n\mapsto \mathbf{S}_t(f_n+g_n)$ and  $f_n\mapsto \mathbf{S}_t(f_n)$ of the Cauchy problem \eqref{CH} satisfies that for $0<t\ll1$
\bbal
\liminf_{n\rightarrow \infty}\left\|\mathbf{S}_{t}(f_n+g_n)-\mathbf{S}_{t}(f_n)\right\|_{C^1}\gtrsim t.
\end{align*}
\end{proposition}
\begin{proof}
Let $u^n_0:=f_n+g_n$. We decompose the solution maps as follows
\bbal
\mathbf{S}_{t}(u^n_0)&=\underbrace{\mathbf{S}_{t}(u^n_0)-u^n_0+tu^n_0\pa_xu^n_0-\int^t_0\mathbf{P}\f(\mathbf{S}_{\tau}(u_0^n)\g)\dd \tau}_{=:\,\mathbf{E}(u_0^n,\,t)}\\
&\quad+f_n+g_n-t(f_n+g_n)\pa_x(f_n+g_n)+\int^t_0\mathbf{P}\f(\mathbf{S}_{\tau}(u_0^n)\g)\dd \tau,\\
\mathbf{S}_{t}(f_n)&=\underbrace{\mathbf{S}_{t}(f_n)-f_n+tf_n\pa_xf_n-\int^t_0\mathbf{P}\f(\mathbf{S}_{\tau}(f_n)\g)\dd \tau}_{=:\,\mathbf{E}(f_n,\,t)}\\
&\quad+f_n-tf_n\pa_xf_n+\int^t_0\mathbf{P}\f(\mathbf{S}_{\tau}(f_n)\g)\dd \tau,
\end{align*}
using the triangle inequality, we deduce that
\bal\label{h0y}
\|\mathbf{S}_{t}(f_n+g_n)-\mathbf{S}_{t}(f_n)\|_{C^1}
&\geq t\f\|g_n\pa_xf_n\g\|_{C^1}-t\f\|g_n\pa_xg_n,\,f_n\pa_xg_n\g\|_{C^1}-\|g_n,\,\mathbf{E}(u_0^n,\,t),\,\mathbf{E}(f_n,\,t)\|_{C^1}\nonumber\\
&\quad-\int_0^t\f\|\mathbf{P}\f(\mathbf{S}_{\tau}(u_0^n)\g)-\mathbf{P}\f(\mathbf{S}_{\tau}(f_n)\g)\g\|_{C^1}\dd\tau.
\end{align}
It is not difficult to find that
\begin{align}\label{yyy1y}
\|g_n\|_{C^1}+\f\|f_n\partial_xg_n\g\|_{C^1}+\f\|g_n\partial_xg_n\g\|_{C^1}&\leq C2^{-n}.
\end{align}
Using Lemma  \ref{lmm3} with $u_0=f_n$ and $u_0=f_n+g_n$, respectively, and then by Lemma \ref{lmm1}, we obtain
\bal\label{yy3y}
\f\|\mathbf{E}(f_n,\,t),\,\mathbf{E}(u_0^n,\,t)\g\|_{C^1}\les t^2.
\end{align}
By Lemmas \ref{lmm1}-\ref{lmm2}, we obtain
\begin{align*}
\f\|\mathbf{P}\f(\mathbf{S}_{\tau}(u_0^n)\g)-\mathbf{P}\f(u_0^n\g)\g\|_{C^1}&\les \f\|\f(\mathbf{S}_{\tau}(u_0^n)\g)^2-(u_0^n)^2\g\|_{L^\infty}+\f\|\f(\pa_x\mathbf{S}_{\tau}(u_0^n)\g)^2-(\pa_xu_0^n)^2\g\|_{L^\infty}\\
&\les \f\|\mathbf{S}_{\tau}(u_0^n)-u_0^n\g\|_{C^1}\f\|\mathbf{S}_{\tau}(u_0^n),\,u_0^n\g\|_{C^1}\les\tau,\\
\f\|\mathbf{P}\f(\mathbf{S}_{\tau}(f_n)\g)-\mathbf{P}\f(f_n\g)\g\|_{C^1}
&\les \f\|\mathbf{S}_{\tau}(f_n)-f_n\g\|_{C^1}\f\|\mathbf{S}_{\tau}(f_n),\,f_n\g\|_{C^1}\les\tau,\\
\f\|\mathbf{P}\f(u_0^n\g)-\mathbf{P}\f(f_n\g)\g\|_{C^1}
&\les \f\|g_n\g\|_{C^1}\f\|f_n,\,g_n\g\|_{C^1}\les 2^{-n},
\end{align*}
which implies that
\begin{align*}
\int_0^t\f\|\mathbf{P}\f(\mathbf{S}_{\tau}(u_0^n)\g)-\mathbf{P}\f(\mathbf{S}_{\tau}(f_n)\g)\g\|_{C^1}\dd\tau
&\les t^2+t2^{-n}.
\end{align*}
Inserting the above \eqref{yyy1y}-\eqref{yy3y} into \eqref{h0y}, we deduce  that
\bbal
\liminf_{n\rightarrow \infty}\|\mathbf{S}_t(f_n+g_n)-\mathbf{S}_t(f_n)\|_{C^1}\gtrsim t\quad\text{for} \ t \ \text{small enough}.
\end{align*}
This finishes the proof of Proposition \ref{prop}. Thus we complete the proof of Theorem \ref{th1}.
\end{proof}

\section{Appendix}
We would like to mention that, the above method in Section \ref{sec4} seems to be invalid for the periodic Camassa-Holm equation since the key fact $\partial_x^{-1}f(x)=\int_0^xf(y)\dd y$ with $f\in L^1(\R)$ have been used. In this section,  motivated by the idea of the traveling wave solutions \cite{Himonas2005}, we shall prove the non-uniform continuity of the data-to-solution map for the periodic Camassa-Holm equation in $C^{1}$.

If $u(x,t) =f(x-t)$ is to be a solution to the CH equation, the
function $f$ must satisfy the ordinary differential equation
\begin{align*}
-f'+f'''+3ff'-2f'f''-ff'''=0.
\end{align*}
Integrating this equation twice gives
\begin{align}  \label{2.3}
(1-f)(f')^2=-f^3+f^2+af+b,
\end{align}
where $a$ and $b$ are arbitrary constants. Since we are looking for a $2\pi$-periodic solution
whose minimum is m and maximum is $M$, we find
\begin{align*}
a=M^2+mM+m^2-M-m,\quad b=mM(1-m-M).
\end{align*}
Equation \eqref{2.3} takes now the form
\begin{align}   \label{2.5}
(1-f)(f')^2=(f-m)(M-f)\big(f-(1-m-M)\big).
\end{align}
We will assume that $0\ll m<M<1.$ More precisely, given suitably small $\delta>0$, we let
\begin{align*}
M=1-\delta,\quad m=1-(\delta+\delta^2).
\end{align*}
If we set
\begin{align}  \label{2.7}
y=1-f,
\end{align}
then \eqref{2.5} becomes
\begin{align}  \label{2.8}
(y')^2=\frac{(\delta+\delta^2-y)(y-\delta)(2-2\delta-\delta^2-y)}y.
\end{align}
In fact, it can be seen that the differential equation \eqref{2.8} admits a nonconstant solution of period 2$\ell$, for some $\ell>0$, which solves the following second-order initial value problem:
\begin{equation}  \label{2.9}
\begin{cases}
y^{\prime\prime}=y-1+\frac{\delta(\delta+\delta^2)(2-2\delta-\delta^2)}{2y^2},\\
y(0)=\delta,\quad y^{\prime}(0)=0.
\end{cases}
\end{equation}
The next lemma gives precise estimates of the solution $y$ and its period.
\begin{lemma}[\cite{Himonas2005}]\label{lmmm}
 For any $0<\delta<1/5$, there exist a positive number $\ell=\ell(\delta)$ and an even $2\ell$-periodic smooth function $y = y(x)$ which solves \eqref{2.9}.
Moreover, the function $y$ satisfies
\begin{align}  \label{2.11}
\delta\leq y(x)\leq\delta+\delta^2
\end{align}
and
\begin{align}  \label{2.1-}
|y'(x)|\leq \sqrt{2}\delta^{\fr32}.
\end{align}
The period $2\ell$ of the solution $y$ depends continuously on the parameter $\delta$ and satisfies
\begin{align}  \label{2.12}
\frac{\sqrt{2}}6\sqrt{\delta+\delta^2}\leq\ell\leq4\sqrt{\delta+\delta^2}.
\end{align}
and the function $u(x,t)=f(x-t)$, where $f(x) = 1 - y(x)$ is a traveling wave solution of the Camassa-Holm equation.
\end{lemma}
We notice  the fact that $\ell$ is continuous with respect to the parameters $\delta$ assures that for any sufficiently large integer $n$, we can always find an $0<\delta<1/5$ such that
\begin{align}  \label{2.17}
\ell=\frac\pi n.
\end{align}
Let $f_n$ be the $2\pi/n$ periodic smooth solution constructed in Lemma \ref{lmmm}.
Define two sequences of travelling wave solutions
\begin{align}  \label{4.2}
u^1_n(x,t)=f_n(x-t)\quad\text{and}\quad\quad u^2_n(x,t)=c_nf_n\left(x-c_nt\right)\quad\text{with}\quad c_n=1+\frac{1}{n}.
\end{align}
Firstly, we need to show that the solution $(u^1_n,u^2_n)$ is uniformly bounded in $C^{1}(\mathbb{T})$. In fact,
\begin{align*}
\|u^1_n(x,t),\,u^2_n(x,t)\|_{C^{1}(\mathbb{T})}\leq 3\left(1+\|y_n\|_{C^{1}(\mathbb{T})}\right)\les1.
\end{align*}
Next by the boundedness of the above results, we have at $t=0$
\begin{align*}
\|u^1_n(x,0)-u^2_n(x,0)\|_{C^{1}(\mathbb{T})}=\frac{1}{n}\left\|1-y_n\right\|_{C^{1}(\mathbb{T})}\to0,\quad n\to\infty.
\end{align*}
Let $t\in(0,T].$ It is found from $\ell=\frac\pi n$ in \eqref{2.17} and \eqref{2.12} that
\begin{align*}
\frac tn=\frac{t\ell}{\pi}\geq\frac{\delta t}{3\pi}.
\end{align*}
Applying the above estimate yields
\begin{align*}
y_n'\left(\dfrac{t}{n}\right)&=y_n'(0)+y_n''(0)\dfrac{t}{n}+o\left(\dfrac{t}{n}\right)\\
&\geq\dfrac{\delta(2-3\delta-\delta^2)}{2}\cdot\dfrac{t}{n}\\
&\geq\dfrac{\delta^2(2-3\delta-\delta^2)) t}{6\pi},
\end{align*}
where we have used the facts from \eqref{2.9}
$$
y_n'(0)=0\quad \text{and}\quad y_n''(0)=\dfrac{\delta(2-3\delta-\delta^2)}{2}.
$$
Taking value at $x=c_nt=(1+1/n)t$ and recalling $y_n^\prime(0)=0$, it follows that
\begin{align*}
\left\|u^1_n(t)-u^2_n(t)\right\|_{C^1(\mathbb{T})}&\geq\left\|\partial_x(u^1_n(t)-u^2_n(t))\right\|_{L^\infty(\mathbb{T})}\\
&=\left\|y'(x-t)-c_ny'(x-c_nt)\right\|_{L^\infty(\mathbb{T})}\\
&\geq\left|y_n^{\prime}\left(\frac{t}{n}\right)\right|\\
&\geq \dfrac{\delta^2(2-3\delta-\delta^2)) t}{6\pi}.
\end{align*}
Combing the above, we obtain the non-uniform continuity of the data-to-solution map for the Camassa-Holm equation on the circle.
\section*{Acknowledgments}
Y. Yu is supported by National Natural Science Foundation of China (12101011).

\section*{Declarations}
\noindent\textbf{Data Availability} No data was used for the research described in the article.

\noindent\textbf{Conflict of interest}
The authors declare that they have no conflict of interest.


\begin{thebibliography}{99}
\linespread{0}\addtolength{\itemsep}{-1.0ex}

\bibitem{Anco2015} S.C. Anco, P.L. da Silva, I.L. Freire, A family of wave-breaking equations generalizing the Camassa-Holm and Novikov equations, J. Math. Phys., 56 (2015), 091506.

\bibitem{Camassa1993} R. Camassa, D. Holm, An integrable shallow water equation with peaked solitons, Phys. Rev. Lett., 71 (1993), 1661-1664.
\bibitem{Chen2015} D. Chen, Y. Li, W. Yan, On the  Cauchy problem for a generalized Camassa-Holm equation, Discrete Contin. Dyn. Syst., 35(3) (2015), 871-889.
\bibitem{Constantin} A. Constantin, Existence of permanent and breaking waves for a shallow water equation: a geometric approach, Ann. Inst. Fourier, 50 (2000), 321-362.

\bibitem{Constantin-E} A. Constantin, The Hamiltonian structure of the Camassa-Holm equation, Exposition. Math., 15 (1997), 53-85.


\bibitem{Escher2} A. Constantin, J. Escher, Well-posedness, global existence, and blowup phenomena for a periodic quasi-linear hyperbolic equation, Comm. Pure Appl. Math., 51 (1998), 475-504.

\bibitem{Escher3} A. Constantin, J. Escher, Wave breaking for nonlinear nonlocal shallow water equations, Acta Math., 181 (1998), 229-243.


\bibitem{Constantin.Strauss} A. Constantin, W. A. Strauss, Stability of peakons, Comm. Pure Appl. Math., 53 (2000), 603-610.
\bibitem{DP} A. Degasperis, D. Holm, A. Hone, A new integral equation with peakon solutions, Theoret. Math.
Phys. 133 (2002), 1463-1474.
\bibitem{Escher} J. Escher, Y. Liu, Z. Yin, Shock waves and blow-up phenomena for the periodic Degasperis-Procesi equation, Indiana
Univ. Math. J. 56 (2007), 87-177.
\bibitem{Fokas1981} A. Fokas, B. Fuchssteiner, Symplectic structures, their b\"{a}cklund transformation and hereditary symmetries, Phys. D, 4 (1981/1982), 47-66.

\bibitem{Grayshan2013} K. Grayshan, A. Himonas, Equations with peakon traveling wave solutions, Adv. Dyn. Syst. Appl., 8 (2013), 217-232.


\bibitem{Himonas2005} A. Himonas, G. Misio{\l}ek, High-frequency smooth solutions and well-posedness of the
Camassa-Holm equation, Int. Math. Res. Not., 51 (2005), 3135-3151.
\bibitem{Himonas2007} A. Himonas, Misio{\l}ek, G. Ponce, Non-uniform continuity in $H^1$ of the solution map
of the CH equation, Asian J. Math., 11 (2007), 141-150.
\bibitem{Himonas2009} A. Himonas, C. Kenig, Non-uniform dependence on initial data for the CH equation on the line, Differ. Integral
Equ., 22 (2009), 201-224.
\bibitem{Himonas2010} A. Himonas, C. Kenig, Misio{\l}ek, Non-uniform dependence for the periodic CH equation, Comm. Partial Differ.
Equ., 35 (2010), 1145-1162.
\bibitem{Himonas2014} A. Himonas, C. Holliman, The Cauchy problem for a generalized Camassa-Holm equation, Adv. Differ.
Equ., 19 (2014), 161-200.

\bibitem{H-H} A. Himonas, C. Holliman, The Cauchy problem for the Novikov equation, Nonlinearity, 25 (2012), 449-479.

\bibitem{HT} J. Holmes, R.C. Thompson, M. Waldrep, Classical solutions of the generalized Camassa-Holm equation, Adv. Differ. Equ.,  22(5-6) (2017), 339-362.
\bibitem{Home2008} A. Home, J. Wang, Integrable peakon equations with cubic nonlinearity, J Phys A., 41 (2008), 372002: 1-11.
\bibitem{Koch2005} H. Koch, N. Tzvetkov, Nonlinear wave interactions for the Benjamin-Ono equation, Int. Math. Res. Not., 30 (2005), 1833-1847.
\bibitem{Lenells} J. Lenells, Traveling wave solutions of the Degasperis-Procesi equation, J. Math. Anal. Appl., 306 (2005), 72-82.
\bibitem{G02} G. Misio{\l}ek, Classical solutions of the periodic Camassa-Holm equation, Geom. Funct. Anal., 12 (2002), 1080-1104.
\bibitem{20jmfm} J. Li, M. Li, W. Zhu, Non-uniform dependence on initial data for the Novikov equation in Besov spaces, J. Math. Fluid Mech., 22:50 (2020).
\bibitem{20jde} J. Li, Y. Yu, W. Zhu, Non-uniform dependence on initial data for the Camassa-Holm equation in Besov spaces, J. Differ. Equ., 269 (2020), 8686-8700.
\bibitem{21jmfm} J. Li, X. Wu, Y. Yu, W. Zhu, Non-uniform dependence on initial data for the Camassa--Holm equation in the critical Besov space, J. Math. Fluid Mech., 23:36 (2021).
\bibitem{24jde} J. Li, Y. Yu, W. Zhu, Non-uniform dependence on initial data for the Camassa--Holm equation in Besov spaces: Revisited, J. Differ. Equ., 390 (2024), 426-450.

\bibitem{Lundmark2007} H. Lundmark, Formation and dynamics of shock waves in the Degasperis-Procesi equation, J. Nonlinear Sci. 17 (2007), 169-198.
\bibitem{Ni2011} L. Ni, Y. Zhou, Well-posedness and persistence properties for the Novikov equation, J. Differ. Equ., 250 (2011), 3002-3021.

\bibitem{Novikov2009} V. Novikov, Generalization of the Camassa-Holm equation, J. Phys. A 42 (2009), 342002.
\bibitem{Vakhnenko} V.O. Vakhnenko, E.J. Parkes, Periodic and solitary-wave solutions of the Degasperis-Procesi equation, Chaos Solitons Fractals, 20 (2004), 1059-1073.
\bibitem{wu2021} X. Wu, Y. Yu, Y. Xiao, Non-uniform dependence on initial data for the generalized Camassa-Holm-Novikov equation in Besov space, J. Math. Fluid Mech., 23:104 (2021).
\bibitem{Yan2019} K. Yan, Wave breaking and global existence for a family of peakon equations with high order nonlinearity, Nonlinear Anal. Real World Appl., 45 (2019), 721-735.
\bibitem{YinJFA} Z. Yin, Global weak solutions for a new periodic integrable equation with peakon solutions, J. Funct. Anal., 212 (2004), 182-194.

\bibitem{Zhao2014} Y. Zhao, Y. Li, W. Yan, Local well-posedness and persistence property for the generalized Novikov equation, Discrete Contin. Dyn. Syst., 34(2) (2014), 803-820.
\end{thebibliography}
\end{document}